\documentclass[11pt]{amsart}
\usepackage[utf8]{inputenc}
\usepackage{amsmath}
\usepackage{thmtools}
\usepackage[colorlinks]{hyperref}
\usepackage{amsthm}
\usepackage{amssymb}
\usepackage[top=1.5in,left=1.5in,right=1.5in]{geometry}
\usepackage{todonotes}
\usepackage{tikz}
\usepackage{tikz-cd}
\usepackage{mathtools}
\usepackage{enumerate}
\usepackage{amsfonts}
\usepackage{mathrsfs}
\usepackage{cleveref}

\crefname{appendix}{appendix}{appendices}
\Crefname{appendix}{Appendix}{Appendices}

\newtheorem{theorem}{Theorem}
\numberwithin{theorem}{section}
\newtheorem{lemma}[theorem]{Lemma}
\newtheorem{proposition}[theorem]{Proposition}

\theoremstyle{definition}

\theoremstyle{remark}
\newtheorem{remark}[theorem]{Remark}

\Crefname{conjecture}{Conjecture}{Conjectures}

\newcommand{\RR}{\mathbb{R}}

\DeclareMathOperator{\Inv}{Inv}

\newcommand{\fin}{\mathrm{fin}}
\newcommand{\aff}{}

\usetikzlibrary{decorations.markings, arrows.meta}
\tikzcdset{
 arrow style=tikz,
 diagrams={>={Straight Barb[scale=0.8]}},
 marrow/.style={decoration={markings,mark=at position 0.5 with {\arrow{#1}}}, postaction=decorate},
 rmarrow/.style={decoration={markings,mark=at position 0.8 with {\arrow{#1}}}, postaction=decorate}
}

\title{Bounded Bruhat intervals and affine Coxeter groups}
\author{Grant T. Barkley}
\author{Christian Gaetz}
\address[Barkley]{Department of Mathematics, University of Michigan, Ann Arbor, MI.}
\email{{\href{mailto:gbarkley@umich.edu}{gbarkley@umich.edu}}}
\address[Gaetz]{Department of Mathematics, University of California, Berkeley, CA.}
\email{{\href{mailto:gaetz@berkeley.edu}{{gaetz@berkeley.edu}}}}

\date{\today}

\begin{document}
\begin{abstract}
Dyer \cite{Dyer1991BruhatGraph} proved that, for each fixed length $k$, only finitely many isomorphism types of Bruhat intervals occur in finite Coxeter groups. In this short paper, we prove that this result holds also for \emph{affine} Coxeter groups, and moreover that this characterizes the affine groups among all infinite Coxeter groups. We also show that the coefficients of $q$ in the Kazhdan--Lusztig polynomials of a Coxeter group are bounded if and only if the group is finite or affine.
\end{abstract}
%\keywords{}
\subjclass{20F55, 05E10, 06A07}
\maketitle
\section{Introduction}

Bruhat order is of fundamental importance in the geometry of flag varieties and the representation theory of Lie algebras and Hecke algebras, yet some basic aspects of its structure are still poorly understood. For example, although important structural properties of Bruhat intervals are known \cite{BjornerWachs1982Shellability,VermaMobius}, it is not known how to characterize or recognize Bruhat intervals, despite much work on aspects of this problem (see e.g. \cite{BjornerEkedahl2009Shape} and results surveyed in \cite[\S2.7-2.8]{BjornerBrenti2005Combinatorics}). Hoped-for applications of an improved understanding of this problem include to the \emph{Combinatorial Invariance Conjecture (CIC)} for Kazhdan--Lusztig polynomials \cite{KazhdanLusztig1979Representations}, posed by Lusztig (c. 1983) and by Dyer \cite{Dyer1987Thesis}. See \cite{BurrullLibedinskyVillegas2025ShapeClass} for discussion of other problems potentially benefiting from such understanding.

The structure of Bruhat intervals is best understood in the family of \emph{finite} Coxeter groups, where Dyer proved \cite{Dyer1991BruhatGraph} that only finitely many isomorphism types of intervals of each fixed length $k$ occur, answering a question of Bj\"{o}rner \cite{BjornerOrderings}. This has allowed for some direct classification for intervals of low length \cite{Akhmedova2021BruhatIntervalsSmallLengths,Hultman2003LengthFour,Hultman2003Thesis} and for several applications to the CIC
\cite{Incitti2006CombinatorialInvariance,Incitti2007MoreCombinatorialInvariance,EspositoMariettiStella2025Flip}.

Recently, interest has extended to \emph{affine} Coxeter groups (see, e.g. \cite{BurrullLibedinskyPlaza2023AffineA2,BurrullLibedinskyVillegas2025ShapeClass}). Our results extend Dyer's finitude theorem to affine type and, conversely, show that it fails outside of these cases. We also derive a consequence for boundedness of Kazhdan--Lusztig polynomial coefficients. We state our results for irreducible Coxeter systems, but their implications for arbitrary systems are straightforward. 

For $u \leq v \in W$, we write $\ell(u,v) \coloneqq \ell(v)-\ell(u)$, where $\ell$ is the length function on $W$; we say the interval $[u,v]$ has length $\ell(u,v)$.

\begin{theorem}
\label{thm:finitely-many-iso}
Fix $k \in \mathbb{N}$. Then only finitely many isomorphism types of Bruhat intervals of length $k$ occur in affine Coxeter groups.
\end{theorem}

We also show that the finite and affine groups are exactly the irreducible Coxeter groups for which the conclusion of \Cref{thm:finitely-many-iso} holds.

\begin{theorem}
\label{thm:infinitely-many-iso}
Let $W$ be an irreducible, infinite Coxeter group which is not affine. Then there exists $k \in \mathbb{N}$ such that $W$ contains infinitely many isomorphism types of Bruhat intervals of length $k$.
\end{theorem}

Lastly, in \Cref{thm:linear-coefficients}, we draw conclusions from our methods for the boundedness of linear coefficients of Kazhdan--Lusztig polynomials $P_{u,v}(q)$. These linear coefficients $[q]P_{u,v}$ and the closely related invariants $d_{u,v}$ are of much recent interest. See, for example, work on the Combinatorial Invariance Conjecture \cite{Patimo2021CoefficientQ}, large hypercubes in Bruhat graphs \cite{EllenbergLibedinskyPlazaSimentalWilliamson2026Hypercubes}, and interpretations in terms of the cohomology of open Richardson varieties, the cluster structure on these, and extensions of Verma modules in Category $\mathcal{O}$ (see \cite{BarkleyGaetzLam2026QCoefficient} for further discussion).

\begin{theorem}
\label{thm:linear-coefficients}
Let $W$ be an irreducible Coxeter group. Then the following are equivalent:
\begin{itemize}
    \item[(i)] the set $\{[q]P_{u,v} \mid u,v \in W\}$ is bounded; and
    \item[(ii)] $W$ is finite or affine.
\end{itemize}
\end{theorem}

\section{Proofs of results}

\subsection{Key observations}
Dyer's proof in \cite{Dyer1991BruhatGraph} of the finitude of length-$k$ intervals in finite Coxeter groups applied his theory of \emph{reflection subgroups} \cite{Dyer1990ReflectionSubgroups}, which implies that such an interval is realized in some finite group of rank at most $k$, of which there are finitely many. Using this same approach in the affine case, however, we run into the obstacle that these reflection subgroups themselves may be infinite, so \emph{a priori} might contain infinitely many types of length-$k$ intervals. We overcome this with two observations: \Cref{prop:affine-cover-bound}, which is special to affine type, bounds the number of upper or lower covers of any element and \Cref{lem:no-K23}, which holds in all Coxeter groups, bounds the growth of rank sizes in a Bruhat interval.

We write $\Phi$ for a root system associated to a Coxeter group $W$, $\Phi^+$ for a choice of positive roots, and $\Phi^-$ for $\Phi \setminus \Phi^+$. We will make use the following result, due to Dyer in \cite[Prop.~1.4]{DyerCones}. We give a different proof in \Cref{app:alternate-proof} that may be of independent interest.

\begin{proposition}
\label{prop:Dyer-extreme}
Let $W$ be a Coxeter group with root system $\Phi$. For $w \in W$ define $\Inv(w)=\{\alpha \in \Phi^+ \mid w^{-1}\alpha \in \Phi^-\}$. Define cones $C_w^+$ and $C_w^-$ in $\mathbb{R}\Phi$ as the nonnegative spans of $\Phi^+ \setminus \Inv(w)$ and of $\Inv(w)$, respectively. Then $C_w^+$ and $C_w^-$ have extreme ray sets equal to $\{\RR_{\geq 0}\alpha \mid w \lessdot t_{\alpha}w \}$ and $\{\RR_{\geq 0} \alpha \mid t_{\alpha}w \lessdot w\}$, respectively.
\end{proposition}

\begin{proposition}
\label{prop:affine-cover-bound}
Let $W_{\aff}$ be an affine Coxeter group and let $\Phi_{\fin}$ denote the associated finite root system. Let $U(w) \coloneqq \{x \in W_{\aff} \mid w \lessdot x\}$ and $D(w)\coloneqq \{x \in W_{\aff} \mid x \lessdot w \}$ denote the sets of upper and lower Bruhat covers of $w \in W_{\aff}$, respectively. Then
\[
|U(w)|, |D(w)| \leq |\Phi_{\fin}|.
\]
\end{proposition}
\begin{proof}
Let $\Phi_{\aff} \supsetneq \Phi_{\fin}$ denote the (untwisted) affine root system for $W_{\aff}$ and fix $w \in W_{\aff}$. Let $B \subset \Phi_{\aff}^+$ be the (necessarily distinct) positive affine roots corresponding to the lower cover relations $x \lessdot w$ (resp. upper cover relations $w \lessdot x$). By \Cref{prop:Dyer-extreme}, the roots in $B$ are the extreme rays of a polyhedral cone $C^-_w$ (resp. $C^+_w$) in $\mathbb{R}\Phi_{\aff}$, and thus they are in convex position. The positive affine roots can be expressed as $\Phi_{\aff}^+ = \{\alpha +n\delta \mid \alpha \in \Phi_{\fin}^+, \: n \in \mathbb{Z}_{\geq 0}\} \cup \{-\alpha +n\delta \mid \alpha \in \Phi_{\fin}^+, \: n \in \mathbb{Z}_{> 0}\}$, where $\delta$ is the primitive imaginary root. By convexity, for each $\alpha \in \Phi^+_{\fin}$, at most two elements of $B$ are in $\pm \alpha + \mathbb{Z}_{\geq 0} \delta$. Thus $B$ is finite and $|B| \leq 2 |\Phi^+_{\fin}| = |\Phi_{\fin}|$.
\end{proof}

The \emph{ranks} $R_0,\ldots,R_{\ell}$ of a Bruhat interval $[u,v]$ with $\ell(u,v)=\ell$ are $R_i \coloneqq \{x \in [u,v] \mid \ell(u,x)=i\}$; we write $r_i\coloneqq |R_i|$ for the size of the $i$th rank.

\begin{lemma}
\label{lem:no-K23}
Let $r_0,\ldots,r_{\ell}$ be the rank sizes of a length-$\ell$ Bruhat interval $[u,v]$ in a Coxeter group $W$, for $\ell \geq 2$. Then $r_{i+1} \leq r_i(r_i-1)$ and $r_{i-1} \leq r_{i}(r_{i}-1)$ for $1 \leq i \leq \ell-1$.
\end{lemma}
\begin{proof}
We prove the first inequality, the second being exactly analogous. Let $1 \leq i \leq \ell-1$ and consider the cover relations between $R_i$ and $R_{i+1}$. All elements of $R_{i+1}$ have at least two lower covers in $R_i$, since all length-two Bruhat intervals are diamonds by \cite{Deodhar1977BruhatOrderingMobius}. On the other hand, by \cite[Thm.~3.2]{BrentiCaselliMarietti2006SpecialMatchings} (or as can be seen by reduction to reflection subgroups of rank two) no three distinct elements of $R_{i+1}$ can all cover the same pair of elements of $R_i$. Thus $r_{i+1} \leq 2 \binom{r_i}{2}=r_i(r_i-1)$.
\end{proof}

\begin{remark}
See \cite{BjornerEkedahl2009Shape} for stronger constraints on the rank sizes of \emph{lower} Bruhat intervals $[e,v]$ in crystallographic type.
\end{remark}

\subsection{Proofs of the main theorems} We now apply \Cref{prop:affine-cover-bound} and \Cref{lem:no-K23} to prove \Cref{thm:finitely-many-iso,thm:infinitely-many-iso,thm:linear-coefficients}.

\begin{proof}[Proof of \Cref{thm:finitely-many-iso}]
Fix $k \in \mathbb{N}$, let $W$ be an irreducible affine Coxeter group, and let $u \leq v$ be elements of $W$ with $\ell(u,v)=k$. Let $t_1,\ldots,t_k$ be the reflections labeling the edges of any maximal chain from $u$ to $v$. Then the reflection subgroup $W'=\langle t_1,\ldots,t_k \rangle$ is a Coxeter group by \cite{Dyer1990ReflectionSubgroups} and is a product of finite and affine Coxeter groups of total rank at most $k$ by \cite[Thm.~3]{DyerLehrer2011ReflectionSubgroups}. Moreover, by \cite{Dyer1991BruhatGraph}, there exist $u',v' \in W'$ such that $[u',v']' \cong [u,v]$ as posets, where on the left we use a primed interval to indicate the intrinsic Bruhat order on $W'$.

If $W'$ is reducible or finite, then there are finitely many possibilities for the isomorphism type of $[u',v']'$ by induction on $k$ and Dyer's result \cite[Cor.~2.2]{Dyer1991BruhatGraph} for finite type. Thus we may assume that $W'$ is an irreducible affine Coxeter group of rank at most $k$. Let $r_0,\ldots,r_k$ denote the rank sizes of $[u',v']'$ (or equivalently of $[u,v]$). Then, since there are finitely many finite root systems of a given rank, we have by \Cref{prop:affine-cover-bound} that $r_{k-1} \leq |D(v')| \leq N_k$ where $N_k$ depends only on $k$. Repeated application of \Cref{lem:no-K23} then implies that each of the $r_i$ has an upper bound depending only on $k$. Since there are finitely many graded posets with these bounded rank sizes, the theorem is proved.
\end{proof}

\begin{proof}[Proof of \Cref{thm:infinitely-many-iso}]
Let $W$ be an irreducible, infinite Coxeter group which is not affine. Such a $W$ also has rank at least three, so by \cite[Prop.~6.4]{DyerHohlwegRipoll2016ImaginaryCones} $W$ has a reflection subgroup $W'$ isomorphic to $U_3$, the universal Coxeter group of rank three. Let $s'_1,s'_2,s'_3$ be the reflections in $W$ which are the simple generators in the intrinsic Coxeter structure on $W'$. Since $W'$ is universal, it is clear (as noted in \cite[Ex.~2.7.9]{BjornerBrenti2005Combinatorics}) from the Subword Criterion (see, e.g. \cite[Thm.~2.2.2]{BjornerBrenti2005Combinatorics}) for Bruhat order that for each $m \geq 1$ the interval $[v_{m-1},v_m]'$ is a $3m$-crown (the face poset of a $3m$-gon), where $v_m \coloneqq (s_1's_2's_3')^m$. In particular, we can construct arbitrarily large intervals of $W'$-length three. In terms of the length function on $W$, we have $\ell(v_{m-1},v_m) \leq \ell(s_1')+\ell(s_2')+\ell(s_3')$. Thus, since 
\[
[v_{m-1},v_m]' \subseteq [v_{m-1},v_m]
\]
by \cite[Thm.~1.4]{Dyer1991BruhatGraph}, for some $k \leq \ell(s_1')+\ell(s_2')+\ell(s_3')$ there are arbitrarily large Bruhat intervals of length $k$ in $W$. Since each Bruhat interval is finite, there are infinitely many isomorphism types of intervals in $W$ of this length $k$.
\end{proof}

\begin{proof}[Proof of \Cref{thm:linear-coefficients}]
It is known by \cite{EliasWilliamson2014Hodge} that $[q]P_{u,v} \geq 0$ in any Coxeter group $W$. It is shown in \cite[Main Thm.(i)-(ii)]{Dyer1997CoefficientQ} that $[q]P_{u,v}=c_{u,v} - d_{u,v}$, where $c_{u,v}$ is the number of coatoms of the interval $[u,v]$ and where $d_{u,v}$ is the dimension of a certain vector space $V_{u,v}$. Thus, if $W$ is affine, we have $0 \leq [q]P_{u,v} \leq c_{u,v} \leq |D(v)| \leq |\Phi_{\fin}|$ is bounded on $W$. 

If instead $(W,S)$ is infinite and not affine, consider $v_m$ as defined in the proof of \Cref{thm:infinitely-many-iso}. Dyer (\emph{loc. cit.}) proves that $d_{u,v}$ is always at most the number $a_{u,v}$ of atoms of $[u,v]$. Thus for $u=e$ we have $d_{u,v_m} \leq |S|$, since the atoms of $[e,v_m]$ are a subset of $S$. In particular, the $d_{e,v_m}$ are bounded. On the other hand, we have seen that the intervals $[v_{m-1},v_m]$ grow arbitrarily large while having a bounded length. It thus follows from \Cref{lem:no-K23} that $c_{e,v_m} \geq c_{v_{m-1},v_m}$ grows arbitrarily large. We conclude that $[q]P_{e,v_m}$ is unbounded as $m$ grows.
\end{proof}

\subsection*{Acknowledgments}
GTB was supported by National Science Foundation grant DMS-2503536. CG was partially supported by NSF grant DMS-2452032 and by a travel grant from the Simons Foundation. CG also thanks Yibo Gao for many fruitful conversations about Bruhat order and affine Weyl groups. We thank the anonymous referee for suggesting we add the appendix below.

\subsection*{AI Statement}
No AI was used for this paper.

\makeatletter
\renewcommand{\@seccntformat}[1]{%
  \@nameuse{the#1}\quad}% removes trailing period from section counters
\makeatother
\appendix
\section{}
\label{app:alternate-proof}
In this appendix, we will give a proof of \Cref{prop:Dyer-extreme} using Bruhat interval polytopes. 
We work with a finite rank Coxeter group $W$ with root system $\Phi$ in a real vector space $V$ with associated symmetric bilinear form $(\cdot, \cdot)$. We say an element $\rho\in V$ is a \emph{dominant regular weight} if $(\alpha,\rho)>0$ for all positive roots $\alpha \in \Phi^+$. It is possible that the span of the roots do not contain any dominant regular weight; this happens, for instance, with affine root systems. In this case we can simply replace $V$ with a vector space of dimension one larger, and declare that the new basis element $\rho$ has pairing $(\alpha_i, \rho)=1$ for all simple roots $\alpha_i$. Hence we may always choose $V$ so that there exists a dominant regular weight. From now on we fix a dominant regular weight $\rho$.

Let $[u,v]$ be a Bruhat interval. We define the \emph{Bruhat interval polytope} $\mathsf{Pol}_{u,v}$ to be the convex hull in $V$ of the vectors
    \[ \{x\rho \mid x\in [u,v]\}. \]
Bruhat interval polytopes were introduced in type $A$ in \cite{KodamaWilliams-15-toda}, studied in finite crystallographic type in \cite{TsukermanWilliams-15-polytopes}, and studied in general finite type in \cite{CaselliDAdderioMarietti}. For general Coxeter groups, the following was shown in \cite[Thm.~5.2.2 \& Lems.~5.2.4-5.2.5]{barkleythesis}.

\newcommand{\Pol}{\mathsf{Pol}}
\begin{proposition}\label{prop:bruintpol}
    The following hold, for any Bruhat interval $[u,v]$: 
    \begin{itemize}
        \item[(a)] Each
    vertex of $\Pol_{u,v}$ is of the form $x\rho$ for a unique $x\in [u,v]$.
    \item[(b)] If an edge of $\Pol_{u,v}$ connects $x\rho$ to $y\rho$, then $x\lessdot y$ or $x\gtrdot y$. 
    \end{itemize}
\end{proposition}

\begin{lemma}\label{lem:bruhatrho}
For all $w\in W$ and $\alpha\in \Phi^+$, the vector $ w\rho - t_{\alpha}w\rho$ is a positive multiple of $\alpha$ if $t_\alpha w > w$, and it is a negative multiple of $\alpha$ if $t_\alpha w < w$. 

\end{lemma}
\begin{proof}
    Observe that 
    \[w \rho - t_{\alpha}w\rho = 2\frac{(\alpha, w\rho)}{(\alpha,\alpha)}\alpha,\]
    so the claim is that $(\alpha, w\rho) >0$ if $t_\alpha w > w$, and $(\alpha, w\rho) < 0$ if $t_\alpha w < w$. Note that $(\alpha,w\rho) = (w^{-1}\alpha, \rho)$. By definition, the pairing of a root with $\rho$ is positive if the root is positive and negative if the root is negative. Hence $(\alpha,w\rho)>0$ if $\alpha$ is in $\Phi^+\setminus \Inv(w)$, and $(\alpha,w\rho)<0$ if $\alpha$ is in $\Inv(w)$. This proves the claim. 
\end{proof}

\begin{lemma}\label{lem:cover-is-extreme}
    Suppose that $w \lessdot t_\alpha w$ (resp. $w\gtrdot t_\alpha w$) is a cover relation. Then $\alpha$ spans an extreme ray of $C_w^+$ (resp. $C_w^-$).
\end{lemma}
\begin{proof}
    Assume for contradiction that $\alpha$ is a non-trivial positive combination of positive roots $\beta_1,\ldots, \beta_k$, none of which are inversions of $w$. Then $t_{\beta_i}w>w$ for all $i$.
    By the directedness of Bruhat order (see \cite[Prop.~2.2.9]{BjornerBrenti2005Combinatorics}), there is an element $z$ which is a common upper bound for all of $t_\alpha w,t_{\beta_1}w,\ldots,t_{\beta_k}w$. Hence the interval $[w,z]$ contains all of those elements. It is a well-known property of polytopes that the cone which is the nonnegative span of $\Pol_{w,z}-w\rho$ has extreme rays spanned by the edge directions of $\Pol_{w,z}$ incident to $w\rho$.  Our assumption implies that $\alpha$ is not one of these edge directions, so $t_\alpha w\rho$ is not connected to $w\rho$ by an edge. 
    
    Now consider a linear functional $f:V\to \RR$ which takes positive values on the positive roots. Then for any cover relation $x\lessdot y$, we have $f(x\rho)>f(y\rho)$ by \Cref{lem:bruhatrho}. In particular, the maximum value on $\Pol_{u,v}$ is uniquely taken at $w\rho$. Applying the simplex method to maximize $f$ beginning at the vertex $t_\alpha w\rho$, we find that there must be an edge connecting $t_\alpha w\rho$ to a vertex $x\rho$ with a larger value of $f$. This implies that $t_\alpha w \gtrdot x$, so we must have $x= w\rho$. Hence there is an edge connecting $w\rho$ to $t_\alpha w\rho$, a contradiction. We deduce that $\alpha$ spans an extreme ray of $C_w^+$. 
    
    The case of $C_w^-$ is similar. 
\end{proof}

\begin{lemma}\label{lem:extreme-is-cover}
    Suppose that $w<t_{\alpha}w=x$ (resp. $x=t_{\alpha}w<w$) is not a cover relation. Then $\alpha$ is not an extreme ray of $C_w^+$ (resp. $C_w^-$).
\end{lemma}
\begin{proof}
    We prove the case $w<x$, the other being very similar. By \cite[Thm.~1.4, Prop.~3.3]{Dyer1991BruhatGraph}, there is a rank-two reflection subgroup $W'$, containing $t_{\alpha}$, and $x_0 \in W'x$ such that $wx_0^{-1} < xx_0^{-1}$ is not a cover relation in the Bruhat order on $W'$. An easy check in dihedral groups shows that there are therefore $\alpha \neq \beta, \gamma \in \Phi' \subset \Phi$ with $\alpha \in \mathbb{R}_{\geq 0}\beta + \mathbb{R}_{\geq 0} \gamma$ and $\beta, \gamma \not \in \Inv(wx_0^{-1})$. The element $x_0$ is chosen so that right multiplication by $x_0$ gives an isomorphism of directed Bruhat graphs on $W'$ and $W'x$. Thus $\beta, \gamma \not \in \Inv(w)$, and $\mathbb{R}_{\geq 0} \alpha$ is not an extreme ray of $C_w^+$.
\end{proof}

\Cref{lem:cover-is-extreme,lem:extreme-is-cover} together imply \Cref{prop:Dyer-extreme}.

\bibliographystyle{halpha-abbrv}
\bibliography{affine}
\end{document}